\documentclass[12pt]{article}
\usepackage[utf8]{inputenc}
 \usepackage[margin=1in]{geometry}
\usepackage{scalerel}
\usepackage{enumerate}
\usepackage{listings}

\usepackage{framed}
\usepackage{amsmath}
\usepackage{amssymb}
\usepackage{amsthm}
\usepackage{wrapfig,color}
\usepackage{authblk}
\usepackage{mathtools}
\usepackage{tikz}
\usepackage{pgfplots}
\usepackage{tikz-cd}
\usepackage{caption}
\usepackage{etoolbox}
\usepackage[normalem]{ulem} 
\usepackage[utf8]{inputenc}
\usetikzlibrary{matrix,arrows,decorations.pathmorphing,patterns,shapes.geometric}
\usepackage{algorithm}
\usepackage[noend]{algpseudocode}

\newtheorem{theorem}{Theorem}[section]
\newtheorem{proposition}[theorem]{Proposition}
\newtheorem{corollary}[theorem]{Corollary}
\newtheorem{lemma}[theorem]{Lemma}
\newtheorem{Remark}[theorem]{Remark}
\newtheorem{claim}{Claim}

\theoremstyle{definition}
\newtheorem*{definition}{Definition}
\newtheorem{example}[theorem]{Example}

\usepackage{chngcntr}
\usepackage{apptools}
\AtAppendix{\counterwithin{theorem}{section}}

\definecolor{darkblue}{rgb}{0.0, 0.0, 0.8}
\definecolor{darkred}{rgb}{0.8, 0.0, 0.0}
\definecolor{darkgreen}{rgb}{0.0, 0.8, 0.0}

\newcommand{\diam}{\mathrm{diam}}
\newcommand{\size}{\mathrm{D}}
\newcommand{\dgh}{d_\mathrm{GH}}
\newcommand{\dis}{\mathrm{dis}}
\newcommand{\R}{\mathbb{R}}

\newcommand{\intf}{d_\mathrm{I}^\mathrm{F}}

\newcommand{\intp}{d_\mathrm{I}}
\newcommand{\skeleton}[1]{\mathrm{C}(#1)}
\newcommand{\density}{\delta}

\newcommand{\power}{\mathrm{P}}
\newcommand{\degree}{\mathrm{deg}}
\newcommand{\codegree}{\mathrm{codeg}}
\newcommand{\PH}{\mathrm{PH}}
\newcommand{\VR}{\mathrm{VR}}
\newcommand{\iden}{\mathrm{id}}
\newcommand{\deltainv}{vertex quasi-distance \,}

\begin{document}

\title{ Quantitative Simplification of Filtered Simplicial Complexes\footnote{This work was partially supported by NSF grants IIS-1422400 and  CCF-1526513.}}

\author[1]{Facundo M\'emoli}
\author[2]{Osman Berat Okutan}
	\affil[1]{Department of Mathematics and Department of Computer Science and Engineering, The Ohio State University. \texttt{memoli@math.osu.edu}}
	\affil[2]{Department of Mathematics, The Ohio State University. \texttt{okutan.1@osu.edu}}

\maketitle

\begin{abstract}
We introduce a new invariant defined on the vertices of a given filtered simplicial complex, called \emph{codensity}, which controls the impact of removing vertices on persistent homology.  We achieve this control through the use of an interleaving type of distance between fitered simplicial complexes.  We study the special case of  Vietoris-Rips filtrations and show that our bounds offer a significant improvement over the immediate bounds coming from considerations related to the Gromov-Hausdorff distance.  
Based on these ideas we give an iterative method for the practical simplification of filtered simplicial complexes.

As a byproduct of our analysis we identify a notion of \emph{core} of a filtered simplicial complex which admits the interpretation as a minimalistic simplicial filtration which retains all the persistent homology information. 
 \end{abstract}


\section{Introduction}
 
Topological data analysis tries to combine and take advantage of the quantitative (but albeit often noisy) nature of Data and the qualitative nature of Topology \cite{c09}. This is done through a machinery that assigns a scale dependent family of topological spaces to given dataset and and then studying how topological properties behave as we change the scale. For a subset $I$ of $\R$, a \emph{filtered simplicial complex} indexed over $I$ is a family $(X^t)_{t\in I}$ of simplicial complexes such that for each $t \leq t'$ in $I$, $X^t$ is contained in $X^{t'}$. Filtered simplicial complexes arise in topological data analysis for example as Vietoris-Rips or \v{C}ech complexes of metric spaces \cite{eh10}. Simplicial complexes have the advantage of admitting a discrete description, hence they are naturally better suited for computations when compared to arbitrary topological spaces. 

A useful and computationally feasible way of analyzing the scale dependent features of a filtered  simplicial complex is through persistent homology and persistence diagrams/barcodes \cite{c09,eh10}. Given a filtered  simplicial complex $X^*$, for a given $k\in\mathbb{N}$, efficient computation of its $k$-th dimensional persistent homology $\PH_k(X^*)$ is studied in many papers, for example \cite{elz02, zc05,dfw14,edelsbrunner2014persistent}: Persistent homology can be computed in time cubic in the number of simplices.

Given this computational complexity, in the interest of being able to process large datasets, an important task is that of \emph{simplifying} filtered simplicial complexes (that is, reducing  the total number of total simplices) in a way such that it is possible to precisely quantify the trade-off between degree of simplification and loss/distortion of homological features \cite{elz02,ks13,s13,c15,dfw14,simba,botnan2015approximating}.

In this paper we consider the effect on persistent homology of removing a vertex and all cells containing it. In this respect, our study is related to  \cite[Section 7]{s13} and \cite[Section 6]{c15}.  A standard measure of the change in persistent homology is called the \textit{interleaving distance}, which is, by the Isometry Theorem \cite[Theorem 3.4]{l15}, isometric to the \textit{bottleneck distance} for persistent barcodes. To quantify the distortion at the persistent homology  level incurred  by operations carried out at the simplicial level, we introduce an interleaving type distance for filtered simplicial complexes which is compatible with the distance between their persistent homology signatures. More precisely,  persistent homology is stable with respect to this new metric. We bound the effect of removing a vertex with respect to this new metric in terms of a new invariant that we call the \emph{codensity} of the vertex, which in turn gives a bound on the change in persistent homology. 

\begin{wrapfigure}[19]{r}{0.28\textwidth}
		\vspace*{-0.4in}	
  \includegraphics[width=0.27\textwidth]{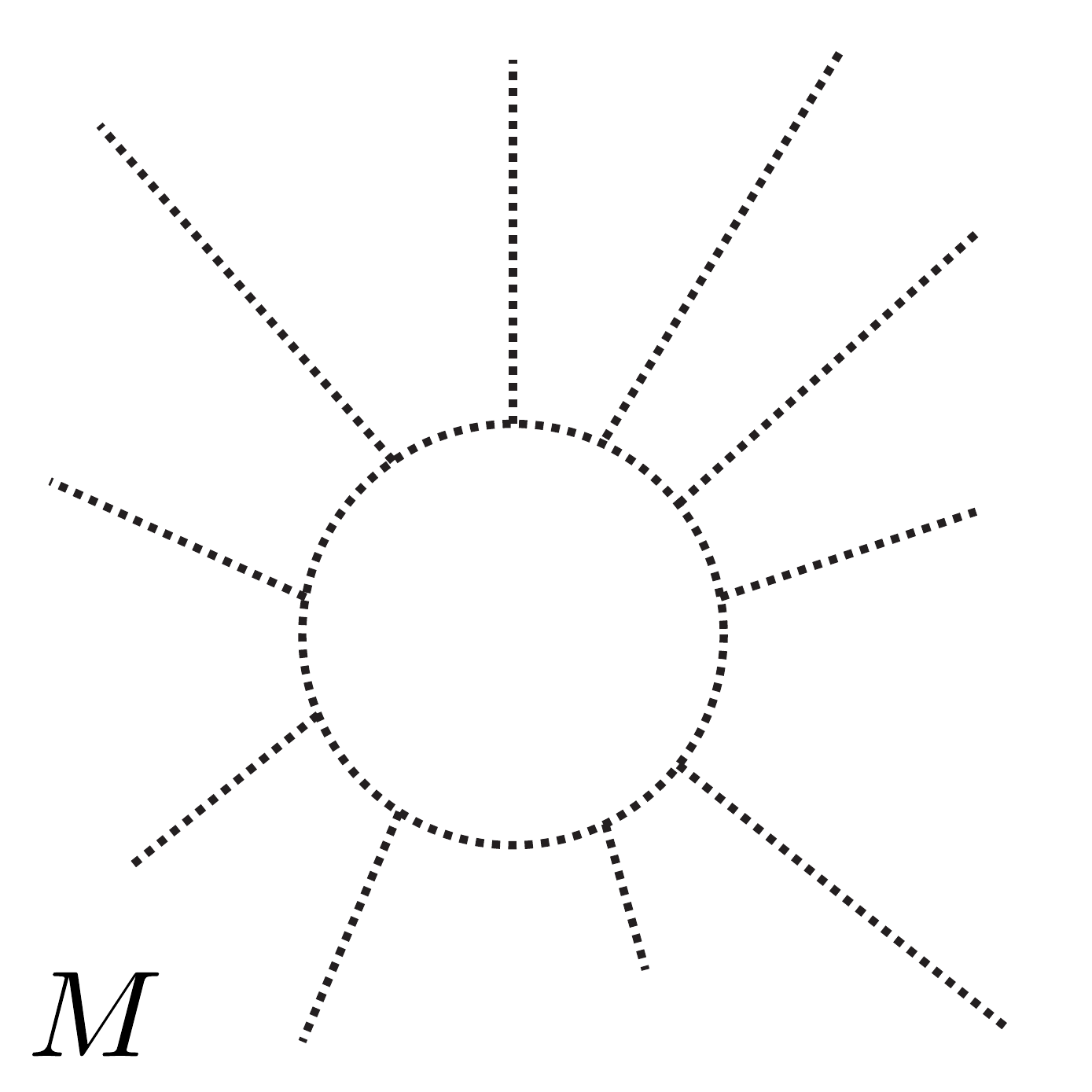}\\
  \includegraphics[width=0.27\textwidth]{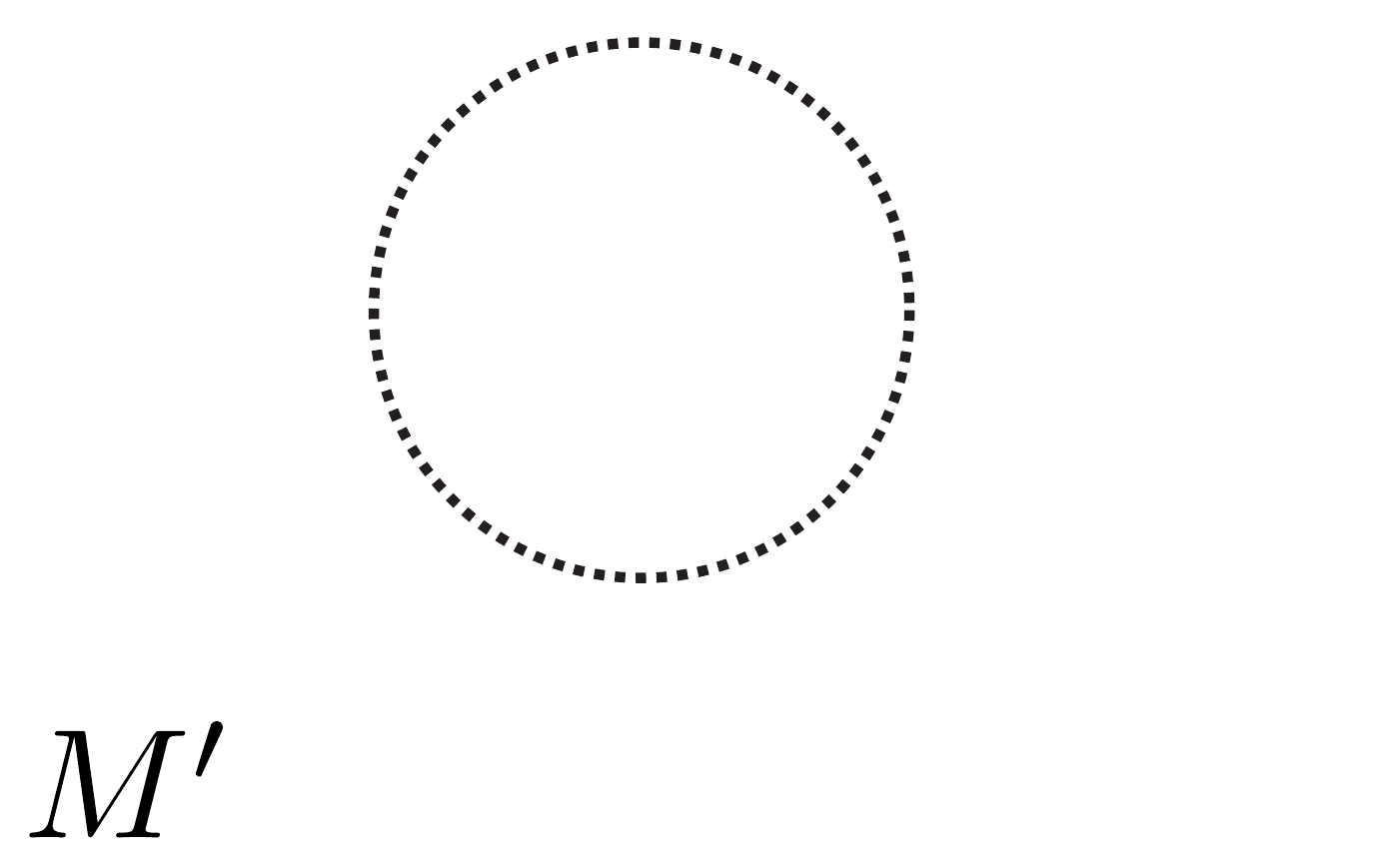}
\caption{These two finite spaces have the same Vietoris-Rips $\PH_{\geq 1}$, see Example \ref{ex:flares}.}
\label{fig:flares}
	\end{wrapfigure}
    
Applying these ideas to Vietoris-Rips complexes of finite metric spaces, in Section \ref{sec:vr} we show in particular that the Vietoris-Rips filtrations of the two finite metric spaces $M$ and $M'$ in Figure \ref{fig:flares} have the same Vietoris-Rips persistent homology in dimensions $1$ and higher: $\PH_{\geq 1}(\mathrm{VR}(M))=\PH_{\geq 1}(\mathrm{VR}(M')).$ This result does not follow from the standard stability of Vietoris-Rips persistence result \cite{dgh-rips}: In fact by increasing the length of the flares in $M$ one can make the Gromov-Hausdorff distance between $M$ and $M'$ grow without bound.

\smallskip
\noindent
\textbf{Contributions and structure of the paper.} For simplicity, in this paper we assume that all families of simplicial complexes are pointwise finite dimensional, indexed over $\R$, and constructible (i.e. changes happen at finitely many indices and the births of cells are realized).  By the functoriality of homology (with coefficients in a field), taking the homology of a filtered simplicial complex yields a persistence module, and hence a persistence barcode. As we have several notions of similarity for barcodes like the bottleneck or Wasserstein distances, a natural question to ask is what type of notions of similarity can be defined for filtered simplicial complexes so that they interact nicely with the desired distance for barcodes. 

We start Section \ref{distance} by reviewing the generalization of the Gromov-Hausdorff distance to filtered simplicial complexes given in \cite{m17}. We show that it generalizes the Gromov-Hausdorff distance between metric spaces in the sense that the Gromov-Hausdorff distance between metric spaces is equal to the Gromov-Hausdorff distance between their Vietoris-Rips complexes, using the ideas in \cite[Proposition 5.1]{m17}. We then introduce an interleaving type pseudo-distance $\intf$ for filtered simplicial complexes. By its categorical nature, interleaving type distances appear in many different settings \cite{chazal12,bmw14,bs14,l15,mbw13,bl17}. It is known that the interleaving distance between the persistent homology of Vietoris-Rips complexes of metric spaces is less than or equal to twice the Gromov-Hausdorff distance between the spaces,  see \cite{dgh-rips} and \cite[Lemma 4.3]{cdo14}. We have the following general theorem:

\begin{theorem}[Stability]\label{stb}
Let $X^*,Y^*$ be constructible filtered simplicial complexes. Then, for every $k\in\mathbb{N}$ we have
$$\intp\big(\PH_k(X^*),\PH_k(Y^*)\big) \leq \intf(X^*,Y^*) \leq 2\, \dgh(X^*,Y^*).$$ 
\end{theorem}
The proof of this theorem is given in Section \ref{stability}. In the construction of the metric $\intf$  we had to pay special attention to the notion of \emph{contiguity}, a coarse way in which homotopy arises between simplicial maps; related studies appear in \cite{bm13,s13,bl17}. 

In Section \ref{simplification}, given a filtered simplicial complex $X^*$, we introduce an invariant $\delta_X(v,w) \geq 0$, called the \emph{vertex quasi-distance of $X^*$}, defined  for each pair of vertices $v$ and $w$. We then define $\density_X(v)$, the \emph{codensity} of the vertex $v$, as the minimal of $\delta_X(v,w)$ as $w$ ranges over all vertices distinct from $v$. We show that this invariant controls the contribution of a vertex to the persistent homology, in a way described in the following proposition:  

\begin{proposition}[Removal of a vertex]\label{removal}
Let $v$ be a vertex of $X^*$ and $(X-\{v\})^*$ be the full filtered subcomplex of $X^*$ obtained by removing the vertex $v$. Then,
$$\intf\big(X^*,(X-\{v\})^*\big) \leq \delta_X(v). $$
\end{proposition}
This proposition shows that by computing $\delta_X(v,w)$ for all $v,w$ we have a method for simplifying a filtered simplicial complex while keeping definite guarantees in terms of the approximation error of the persistent homology. We then discuss how we can make the calculation of $\delta_X(v,w)$ simpler and how to make $\delta_X(v,w)$ smaller if we are only interested in  persistent homology of certain degrees only (e.g. $\mathrm{PH}_1)$. We then show what our constructions correspond for Vietoris-Rips complexes of metric spaces and give an example showing the advantages of our simplification guarantees to those given by the  Gromov-Hausdorff based bounds of \cite{dgh-rips}.

In Section \ref{classification}, we introduce \textit{simple} filtered simplicial complexes: We call a filtered simplicial complex  $X^*$ \textit{simple} if the condensity $\density_X(v)>0$ for each vertex $v$. Proposition \ref{removal} implies that any non-simple filtered simplicial complex can be reduced in size without changing its persistent homology. Then we show that this observation can be strengthened in the following way:

\begin{theorem}[Classification via cores]\label{cls}
For each filtered finite simplicial complex $X^*$, there exists a unique (up to isomorphism) simple filtered complex $\mathrm{C}^*$ such that $\intf(X^*,\mathrm{C}^*)=0$. Furthermore, $\mathrm{C}^*$ is a full subcomplex of $X^*$.
\end{theorem}

Hence simple filtered complexes classify filtered complexes with respect to $\intf$. We denote $\mathrm{C}^*$ described in Theorem \ref{cls} by $\skeleton{X^*} \subseteq X^*$ and call it  the \textbf{core} of $X^*$. 

Theorem \ref{cls}  above can then be interpreted as follows.  \textit{Equivalence} (i.e. $\intf=0$) between filtered simplicial complexes coincides with  isomorphism between their respective cores: Namely $\intf(X^*,Y^*)=0$ if and only if $\skeleton{X^*}$ and $\skeleton{Y^*}$ are isomorphic. In particular this implies that the number of elements in the core is a well defined invariant. More precisely, the core of a filtered simplicial complex $X^*$ coincides with the minimal cardinality filtered simplicial complex at zero $\intf$ distance from $X^*$ (Corollary \ref{coro:min-core}).

We obtain Theorem \ref{cls} as a corollary of a more general statement  (Proposition \ref{gh-intf}) which says that between simple filtered simplicial complexes, for \textit{small enough} distances, $2\dgh$ and $\intf$ coincide and furthermore this coincidence is realized through specific bijective maps. 

In Section \ref{example}, we give a construction depending on a parameter $r \geq 0$ which extends a filtered simplicial complex so that its $\intf$ distance to the original space is $0$, while the $\dgh$ distance is at least $r/2$. This shows that $\intf$ can be much smaller than $\dgh$.

\section{Gromov-Hausdorff and interleaving type distances between filtered simplicial complexes}\label{distance}

Given a finite set $V$ we denote the power set of $V$ minus the empty set by $\power(V)$. Given a metric space $(X,d_X)$ the \emph{diameter} function is defined by $\diam_X:\power(X)\rightarrow \R_+$, where $\sigma\mapsto \max_{x,x'\in \sigma} d_X(x,x').$   By $\overline{\R}$ we will mean the extended reals $\R\cup\{-\infty,+\infty\}.$

\subsection{Gromov-Hausdorff distance between filtered simplicial complexes}\label{sec:dGH}

 We define the \textit{vertex set} of a filtered simplicial complex as the union of the vertex sets of its components (i.e. individual $X^t$'s). 
 
\begin{definition}[Size function]
Given a finite filtered simplicial complex $X^*$ with vertex set $V$ define the \textit{size function} $ \size_X : \power(V)\to \overline{\R}$ as follows: $\size_X(\alpha):=\inf \{r: \alpha \in X^r \}. $\end{definition}
 
Note that if $\alpha$ is not contained in $X^r$ for any $r$ then $\size_X(\alpha)=\infty$ and if it is contained in all $X^r$ then $ \size_X(\alpha)=-\infty$. Also, by the constructibility the condition $\size_X(\alpha)$ is realized as the minimum if it is finite. Note that if $\alpha \subseteq \alpha'$, then $\size_X(\alpha) \leq \size_X(\alpha')$.

\begin{Remark}\label{diameter}
If $X^*$ is the Vietoris-Rips complex of a metric space, then $\size_X \equiv \diam_X$.
\end{Remark}

Conversely, if we have $\size: \power(V) \to \overline{\R}$ monotonic with respect to inclusion, then we can define a filtered simplicial complex $X_\size^*$ with the vertex set $V$ by 
$X_\size^r:= \{\alpha: \size(\alpha)\leq r  \}.$

\begin{Remark}
These constructions are inverses of each other, more precisely $\size \equiv \size_{X_\size}$ and $X^\ast=X^\ast_{\size_X}.$ Hence a filtered simplicial complex is uniquely determined by its size function.
\end{Remark}

We now review a notion of distance between filtered simplicial complexes \cite{m17}.
\begin{definition}[Tripods and distortion]
A \textit{tripod} between $X^*,Y^*$ with vertex sets $V,W$ respectively is a finite set $Z$ with surjective maps $p_X: Z \to V$ and $p_Y: Z \to W$. The \textit{distortion} $\dis(Z)$ of a tripod $(Z,p_X,p_Y)$ is defined by $\max_{\alpha \in \power(Z)} |\size_X(p_X(\alpha))-\size_Y(p_Y(\alpha))|$, 
where the convention $\infty - \infty =0$ is assumed.
\end{definition}

\begin{definition}[Gromov-Hausdorff distance between filtered simplicial complexes] The Gromov-Hausdorff distance between the filtered simplicial complexes $X^\ast$ and $Y^\ast$ is  
\[\dgh(X^*,Y^*):=\frac{1}{2}\inf \left\{\dis(Z): Z \text{ a tripod between } X^*\,\mbox{and}\,Y^* \right\}.\]
\end{definition}

Note that the product of vertex sets with the projection maps gives  a tripod and if the size functions are finite then the distortion of this tripod is finite. Hence, the Gromov-Hausdorff distance between filtered simplicial complexes with finite size functions is finite.

\begin{Remark}\label{correspondence}
Given a tripod $(Z,p_X,p_Y)$, let $R=\{(p_X(z),p_Y(z)):z \in Z\} \subseteq V \times W$. If we denote the projection maps $V \times W \to V,W$ by $\pi_1,\pi_2$, then $(R,\pi_1,\pi_2)$ is a tripod between $X^*,Y^*$. Furthermore, $\dis(Z)=\dis(R)$. Since the vertex sets $V,W$ are assumed to be finite, there are finitely many such $R$'s. Therefore, the infimum in the definition of $\dgh$ is realized.
\end{Remark}

The definition of the Gromov-Hausdorff distance between filtered spaces generalizes the Gromov-Hausdorff distance between metric spaces:
\begin{proposition}[Extension]\label{extension}
Let $M,N$ be finite metric spaces and $X^*,Y^*$ be their Vietoris-Rips complexes respectively. Then $\dgh(M,N)=\dgh(X^*,Y^*). $
\end{proposition}
\begin{proof}
Let $R$ be a correspondence between $M,N$ (i.e. $R \subseteq M \times N$ and $\pi_M(R)=M$, $\pi_N(R)=N$). Note that $R$ can be considered as a tripod between $X^*,Y^*$. By Remark \ref{correspondence}, it is enough to show that the distortion of $R$ as a metric correspondence between $M,N$ is same with the distortion of $R$ as a tripod between $X^*,Y^*$. Let us denote the first one by $\dis^{met}(R)$ and the second one by $\dis^{tri}(R)$.

\begin{claim}
$\dis^{tri}(R) \geq \dis^{met}(R).$
\end{claim}
\begin{proof}
By Remark \ref{diameter}, the size functions of $X^*,Y^*$ are given by the diameter. Hence we have:
\begin{align*}
\dis^{tri}(R) &\geq \max_{(x,y),(x',y') \in R} |\diam_M(x,x')-\diam_N(y,y')|\\
		  &= \max_{(x,y),(x',y') \in R} |d_M(x,x')-d_N(y,y')|\\
          &=\dis^{met}(R).
\end{align*}
\end{proof}

\begin{claim}
$\dis^{tri}(R) \leq \dis^{met}(R).$ 
\end{claim}
\begin{proof}Let $\alpha \in \power(R)$. Let $\alpha \in \power(R)$ such that $$\dis^{tri}(R)=|\diam_M(\pi_M(\alpha))-\diam_N(\pi_N(\alpha))|.$$ Without loss of generality, we can assume that $$\diam_M(\pi_M(\alpha)) \geq \diam_N(\pi_N(\alpha)).$$ Let $x,x'$ be points in $\pi_M(\alpha)$ so that $$\diam_M(\pi_M(\alpha))=d_M(x,x'). $$ There exists points $y,y'$ in $N$ such that $(x,y),(x',y') \in \alpha$. Then we have
\begin{align*}
\dis^{tri}(R) &= \diam_M(\pi_M(\alpha))-\diam_N(\pi_N(\alpha)) \\
			  &= d_M(x,x')-\diam_N(\pi_N(\alpha))\\
              &\leq d_M(x,x')-d_N(y,y')\\
              &\leq \dis^{met}(R).
\end{align*}
\end{proof}
\end{proof}

We also have:
\begin{proposition}\label{ghmetric}
$\dgh$ is a (pseudo-)metric between filtered simplicial complexes.
\end{proposition}
\begin{proof}
Non-negativity and symmetry properties follows from the definition. $\dgh(X^*,X^*)=0$ since the distortion of the \textit{identity tripod} on the vertex set of $X$ is $0$. Let us show the triangle inequality. Let $(Z,p,p')$ be a tripod between $X^*,X'^*$ and $(Z',q',q'')$ be a tripod between $X'^*,X''^*$. Let $Z''$ be the fiber product $$Z''=Z {_{p'}\times_{q'}} Z' .$$ Then $(Z'',p \circ \pi_Z,q'' \circ \pi_{Z'})$ is a tripod between $X^*,X''^*$. Given $\alpha \in \power(Z'')$, we have
\begin{align*}
|\size_X(p \circ \pi_Z(\alpha))- \size_{X''}(q'' \circ \pi_{Z'})| & \leq 
|\size_X(p \circ \pi_Z(\alpha))-\size_{X'}(p' \circ \pi_Z(\alpha))| \\ &\quad +|\size_{X'}(p' \circ \pi_Z(\alpha)) - \size_{X''}(q'' \circ \pi_{Z'}(\alpha))| \\
&=|\size_X(p \circ \pi_Z(\alpha))-\size_{X'}(p' \circ \pi_Z(\alpha))| \\ &\quad +|\size_{X'}(q' \circ \pi_{Z'}(\alpha)) - \size_{X''}(q'' \circ \pi_{Z'}(\alpha))| \\
&\leq \dis(Z)+\dis(Z').
\end{align*}
Since $\alpha \in \power(Z'')$ was arbitrary, $\dis(Z'') \leq \dis(Z')+\dis(Z'')$. Since the tripods $Z,Z'$ were arbitrary, $\dgh(X^*,X''^*) \leq \dgh(X^*,X'^*)+\dgh(X'^*,X''^*)$.
\end{proof}

\begin{definition}[Isomorphism and weak isomorphism]\label{def:iso}
We call two filtered simplicial complexes \textit{isomorphic} if there exists a size preserving bijection between their vertex sets. We call two filtered simplicial complexes \textit{weakly isomorphic} if their Gromov-Hausdorff distance is $0$.
\end{definition}

\begin{Remark}
Isomorphism implies weak isomorphism. This can be seen by taking the graph of the size preserving bijection with the natural projection maps as the tripod.
\end{Remark}

\begin{example}[A pair of non-isomorphic but weakly isomorphic filtered simplicial complexes]
Given a positive integer $n$ and a real number $c$, let $X^*(n,c)$ be the filtered simplicial complex with vertex set $\{1,\dots,n \}$ and the constant size function equal to $c$. Note that for any $n,m\in\mathbb{N}$ (possibly different), the tripod between $X^*(n,c)$ and $X^*(m,c)$ given by the product of their vertex sets has zero distortion. Hence, $\dgh(X^*(n,c),X^*(m,c))=0$ which means that $X^*(n,c)$ and $X^*(m,c)$ are weakly isomorphic.
\end{example}

\subsection{The interleaving type distance $d_\mathrm{I}^\mathrm{F}$ between filtered simplicial complexes}\label{sec:dI}
Here, we introduce an interleaving type of distance between filtered simplicial complexes which interacts nicely with their persistent homology. We use the following notation/terminology: A persistence module (over $\R$) is a family of vector spaces $(V^t)_{t \in \R}$ with linear maps  $f^{t,t'}: V^t \to V^{t'}$ for $t \leq t'$ such that $f^{t,t}=\iden_{V^t}$ and for each $t \leq t' \leq t''$, $f^{t,t''}=f^{t',t''} \circ f^{t,t'}$. By the functoriality of homology, for $k\in\mathbb{N}$, the homology groups $\mathrm{H}_k(X^i)$ of a filtered simplicial complex $X^*$ form a persistence module, where the linear maps are induced by the inclusion $X^t \hookrightarrow X^{t'}$.  This persistence module is called the $k$-th \textit{persistent homology} of $X^*$ and is denoted by $\PH_k(X^*)$.

A \emph{morphism} between filtered simplicial complexes is a function between their vertex sets. 

\begin{definition}[Degree] Let $f$ be a morphism from $X^*$ to $Y^*$.  Given $r \geq 0$, we say that $f$ is $r$-\textit{simplicial} if $\size_Y(f(\alpha)) \leq \size_X(\alpha) + r$ for each $\alpha$. We define the \textit{degree} $\degree(f)$ of $f$ by $$\degree(f):=\inf\{r\geq 0: f \text{ is } r\text{-simplicial} \}. $$
By the constructibility assumption, $f$ is $\degree(f)$-simplicial. 
\end{definition}
Hence the degree of a morphism can be thought as a measure of the failure of the morphism at being simplicial.

\begin{Remark}\label{induce}
If $f: X^* \to Y^*$ is $r$-simplicial, then it induces a morphism $f_*$ from the persistence module $\PH_k(X^*)$ to $\PH_k(Y^{*+r})$, induced by the simplicial maps $X^t \to Y^{t+r}$, $\alpha \mapsto f(\alpha)$.
\end{Remark}

\begin{definition}[Codegree]
Let $f,g$ be morphisms from $X^*$ to $Y^*$. Given $r \geq 0$, we say that $f,g$ are $r$-\textit{contiguous} if $\size_Y(f(\alpha)\cup g(\alpha)) \leq \size_X(\alpha)+r$ for each $\alpha$. We define the \textit{codegree} $\codegree(f,g)$ of $f,g$ by $$\codegree(f,g):= \inf \{r \geq 0: f,g \text{ are } r\text{-contigouous} \}.$$ By the constructibility assumption, $f,g$ are $\codegree(f,g)$-contiguous.
\end{definition}

\begin{Remark}\label{codegree}
Let $f,g:X^* \to Y^*$ be morphisms of filtered simplicial complexes. Then,
\begin{enumerate}
\item $\degree(f)=\codegree(f,f) \leq \codegree(f,g)$.
\item If $f,g$ are $r$-contiguous, then they induce the same maps  $\PH_k(X^*) \to \PH_k(Y^{*+r})$, as the maps $X^t \to Y^{t+r}$ given by $\alpha \mapsto f(\alpha),g(\alpha)$ are contiguous as simplicial maps. 
\item For each morphism $h:Z^* \to X^*$, we have $\codegree(f\circ h,g\circ h) \leq \codegree(f,g)+\degree(h)$.
\item For each morphism $h:Y^* \to Z^*$, we have $\codegree(h \circ f, h \circ g) \leq \codegree(f,g) + \degree(h).$
\end{enumerate}
\end{Remark}

Assume we are given three morphsims $f,g,h$ such that $\codegree(f,g) \leq r$ and $\codegree(g,h) \leq r$. Although it is possible that $\codegree(f,h) > r$, by  part 1. of the remark above, $f,g,h$ induce the same maps $\PH_k(X^*) \to \PH_k(Y^{*+r})$. The following definition is given to capture this type of situations, see Section \ref{sec:strong-int} below.

\begin{definition}
Define $\codegree^\infty(f,g):=\min_{f=f_0,\dots,f_n=g}\max_{i = 1,\dots,n} \codegree(f_{i-1},f_i). $
\end{definition}

\begin{proposition}\label{codegreeinfty}
Let $f,g,h: X^* \to Y^*$ and $f',g':Z^* \to X^*$ be morphisms of filtered simplicial complexes.
\begin{enumerate}
\item $\degree(f)=\codegree^\infty(f,f) \leq \codegree^\infty(f,g) \leq \codegree(f,g)$.
\item If $\codegree^\infty(f,g)\leq r$, then $f,g$ induce the same maps from $\PH_k(X^*) \to \PH_k(Y^{*+r})$.
\item (Ultrametricity) $\codegree^\infty(f,h) \leq \max\big(\codegree^\infty(f,g),\codegree^\infty(g,h)\big)$.
\item $\codegree^\infty(f\circ f',g \circ g') \leq \codegree^\infty(f,g)+\codegree^\infty(f',g')$.
\end{enumerate}
\end{proposition}
\begin{proof}
\noindent 1. $\codegree^\infty(f,g) \leq \codegree(f,g)$ can be seen by taking $f=f_0,f_1=g$. $\degree(f) \leq \codegree(f,g)$ since for any $f=f_0,\dots,f_n=g$, $\deg(f) \leq \codegree(f^0,f^1)$. This also shows that $\degree(f)=\codegree^\infty(f,f) \leq \codegree^\infty(f,g)$.

\noindent
2. There exists $f=f_0,\dots,f_n=g$ such that $\codegree(f_{i-1},f_i) \leq r$. By Remark \ref{codegree}, $f_{i-1},f_i$ induce the same maps from $\PH(X^*)$ to $\PH(Y^{*+r})$. Hence $f_0=f,g=f_n$ also induce the same maps.

\noindent
3. Follows by concatenating sequences of functions.

\noindent
4. Let $f=f_1,\dots,f_n=g$ be the sequence realizing $\codegree^\infty(f,g)$. Then, for any morphism $h$ whose range is same with the domain of $f,g$, by Remark \ref{codegree} we have
\begin{align*}
\codegree^\infty(f\circ h, g\circ h) &\leq \max_i \codegree(f_{i-1} \circ h, f_i \circ h) \\
&\leq \max_i \codegree(f_{i-1},f_i) + \degree(h) \\
&=\codegree^\infty(f,g)+\degree(h)
\end{align*}
Similarly we have $$\codegree^\infty(h \circ f, h \circ g) \leq \codegree^\infty(f,g)+deg(h). $$
Now by using these and part i),iii) above, we get
\begin{align*}
\codegree^\infty(f \circ f', g \circ g') &\leq \max(\codegree^\infty(f \circ f',g \circ f'),\codegree^\infty(g \circ f', g \circ g')) \\
&\leq \max (\codegree^\infty(f,g)+\degree(g'),\codegree^\infty(f',g')+\degree(g))\\
&\leq \codegree^\infty(f,g)+\codegree^\infty(f',g').
\end{align*}
\end{proof}

\begin{definition}[Interleaving distance between filtered simplicial complexes]\label{def:intf}
For $\epsilon\geq 0$, an $\epsilon$-interleaving between $X^*$ and $Y^*$ consists of morphisms $f: X^* \to Y^*$, $g: Y^* \to X^*$ such that $$\degree(f),\,\degree(g) \leq \epsilon, \,\,\mbox{and}\,\, \codegree^\infty(g \circ f,\iden_{X^*}), \,\codegree^\infty(f \circ g,\iden_{Y^*}) \leq 2\epsilon.$$
In this case we say that $X^*,Y^*$ are $\epsilon$-interleaved. We define $$\intf(X^*,Y^*):=\inf \{\epsilon \geq 0: X^*,Y^* \text{ are } \epsilon \text{-interleaved} \}. $$
\end{definition}

We then have:
\begin{proposition}\label{interleavingdistance}
$\intf$ is a (pseudo-)distance between filtered simplicial complexes.
\end{proposition}
\begin{proof}
Non-negativity and symmetry follow from the definition. $\intf(X^*,X^*)=0$ since $\iden_{X^*}$ gives a $0$-interleaving. Let us show the triangle inequality. Let $f:X^*\to Y^*,g:Y^* \to X^*$ be an $\epsilon$-interleaving between $X^*,Y^*$ and $f':Y^* \to Z^*,g': Z^* \to Y^*$ be an $\epsilon'$-interleaving between $Y^*,Z^*$. Let us show that $f' \circ f, g \circ g' $ is an $(\epsilon+\epsilon')$-interleaving between $X^*,Z^*$. $$\degree(f'\circ f),\degree(g \circ g') \leq \epsilon + \epsilon',$$ which follows from the definition of degree. By Remark \ref{codegreeinfty} we have
\begin{align*}
\codegree^\infty(g\circ g' \circ f' \circ f, \iden_{X^*}) &\leq \max(\codegree^\infty(g \circ g' \circ f' \circ f, g \circ f),\codegree^\infty(g \circ f,\iden_{X^*})) \\
&\leq \max(\degree(g)+\codegree^\infty(g' \circ f',\iden_{Y^*})+\degree(f),2\epsilon) \\
&\leq \max(\epsilon+2\epsilon'+\epsilon,2\epsilon)=2(\epsilon+\epsilon').
\end{align*}
Similarly $$\codegree^\infty(f'\circ f \circ g \circ g',\iden_{Z^*}) \leq 2(\epsilon+\epsilon'). $$ This completes the proof.
\end{proof}

\begin{definition}
We call $X^*,Y^*$ \textit{equivalent} if $\intf(X^*,Y^*)=0$.
\end{definition}

Because of Theorem \ref{stb}, equivalent filtered simplicial complexes have the same persistent homologies. In the next section we see that weakly isomorphic filtered simplicial complexes  (See Definition \ref{def:iso}) are equivalent. For now, let us give an example to show that the converse is not true.

\begin{example}[Equivalence is weaker than weak isomorphism]\label{simplex-star}
Define $\Delta_n^*$ as the filtered simplicial complex with  vertex set $\{0, \dots, n\}$ and  size function $\size_n(\alpha):=\max \{i: i\in \alpha \}$. Note that for any tripod $(R,p,q)$ between $\Delta_n^*,\Delta_m^*$ we have $\dis(R) \geq |\size_n(p(R))-\size_m(q(R))|=|m-n|,$
hence $\dgh(\Delta_m^*,\Delta_n^*) \geq |m-n|/2.$
We now show that $\intf(\Delta_m^*,\Delta_n^*)=0.$ The topological basis of this is the fact that any two maps onto a simplex are contiguous. Without loss of generality assume that $m \leq n$. Let $\iota: \Delta_m^* \to \Delta_n^*$ be the morphism given by the inclusion of the vertex set and let $\pi: \Delta_n^* \to \Delta_m^*$ be the map given by $k \mapsto \min(m,k)$. Since both maps are size non-increasing and size functions are defined by the maximum, both maps have degree $0$. Also note that (1) $\pi \circ \iota=\iden$, and (2) if $\alpha \subseteq \{0,\dots,n \}$ has maximal element $i$, then so does $\alpha \cup \iota \circ \pi (\alpha)$. Hence, $\size_n(\alpha)=\size_n(\alpha \cup \iota \circ \pi (\alpha)).$ This shows that $\codegree(\iota \circ \pi,\iden)=0$. Therefore $\intf(\Delta_m^*,\Delta_n^*)=0.$ In Section \ref{example} we give a construction generalizing this one (see Remark \ref{simplex-generalization}).  
\end{example}

\subsection{About the definition of $\intf$.} \label{sec:strong-int}

It is possible \cite{mm,dowker} to define a related but \emph{strictly stronger} notion of $\epsilon$-interleaving  between filtered simplicial complexes than the one given in Definition \ref{def:intf}. Given  filtered simplicial complexes $X^*$ and $Y^*$, an $\epsilon$-strong interleaving between $X^*,Y^*$ is a pair of morphisms $f: X^* \to Y^*$, $g: Y^* \to X^*$ such that $$\degree(f),\,\degree(g) \leq \epsilon, \,\,\mbox{and}\,\, \codegree(g \circ f,\iden_{X^*}), \,\codegree(f \circ g,\iden_{Y^*}) \leq 2\epsilon.$$

The difference with Definition \ref{def:intf} is that $\codegree^\infty$ has been replaced  by (the generally larger number) $\codegree$. The problem with this definition is that it does not give a metric as we show next.  Define $\widehat{\intf}(X^*,Y^*)$ as the infimal $\epsilon\geq 0$ such that $X^*$ and $Y^*$ are $\epsilon$-strongly interleaved.

Note that the definition of $\codegree^\infty$ uses chains of morphisms. Such a sequence of morphisms used in the proof of Proposition \ref{interleavingdistance} to show the triangle inequality for $\intf$. The topological basis of the necessity of considering chains is the following:  If simplicial maps $f,f': S \to T$ are contiguous and $g,g':T \to U$ are contiguous, it does not necessarily follow  that $g\circ f,g'\circ f': S \to U$ are contiguous. Instead, what we have is $g \circ f$ is contiguous to $g \circ f'$ which is in turn contiguous to $g' \circ f'$. Note that contiguity is not an equivalence relation between simplicial morphisms. 

\smallskip
\noindent
\textbf{$\widehat{\intf}$ is not a metric.} Let us give a concrete example to show that $\widehat{\intf}$ does not satisfy the triangle inequality. For a non-negative integer $n$, let $X_n^*$ be the filtered simplicial complex with  vertex set $\{v_0,\dots,v_n\}$ and such that (1) the cells of  $X_n^0$ coincide with the set of all  edges of the form $[v_i,v_{i+1}]$, (2) $X_n^t$ is the full simplex for $t\geq 1$, (3) $X_n^t=\emptyset$ for $t < 0$, and (4) $X_n^t=X_n^0$ for $0 \leq t < 1$. Note that $X_n^*$ is included in $X_{n+1}^*$ via the morphism $v_i \mapsto v_i$ for all $i=0,\dots,n$. Also, $X_{n+1}^*$ surjects onto $X_n^*$ via the morphism $v_i \mapsto v_i$ for $i=0,\dots,n$ and $v_{n+1} \mapsto v_n$. By using these maps, we see that $\widehat{\intf}(X_n^*,X_{n+1}^*)$ is $0$. However, for $n\geq 3$, $\widehat{\intf}(X_n^*,X_0^*)$ is not $0$, as no constant map from $X_n^0$ to itself is contiguous to the identity.  Therefore, $\widehat{\intf}$  fails to satisfy the triangle inequality, for otherwise one would have $0<\widehat{\intf}(X_0^*,X_3^*)\leq \widehat{\intf}(X_0^*,X_1^*)+\widehat{\intf}(X_1^*,X_2^*)+\widehat{\intf}(X_2^*,X_3^*) = 0,$ a contradiction.

\subsection{Stability}\label{stability}
In this section we prove Theorem \ref{stb}. We first need the following
\begin{lemma}\label{dis-int}
Let $f: X^* \to Y^*, g: Y^* \to X^*$ be morphisms. Let $R$ be the correspondence between the vertex sets of $X^*,Y^*$ containing the graphs of $f$ and $g$. Then $(f,g)$ is an $\dis(R)$-interleaving.
\end{lemma}
\begin{proof}
Let $p_X$ (resp. $p_Y$) be the projection map from $R$ to the vertex set of $X^*$ (resp. $Y^*)$. Let $\alpha$ be a non-empty subset of the vertex set of $X^*$. Note that $$f(\alpha) \subseteq p_Y(p_X^{-1}(\alpha)). $$ Let $\epsilon:=\dis(R)$. We have
\begin{align*}
\size_{Y^*}(f(\alpha)) &\leq \size_{Y^*}(p_Y(p_X^{-1}(\alpha)))\\
					   &\leq \size_{X^*}(p_X(p_X^{-1}(\alpha)))+\epsilon \\
                       &=\size_{X^*}(\alpha)+\epsilon.
\end{align*}
Hence $\degree(f) \leq \epsilon$. Similarly $\degree(g) \leq \epsilon$.

Note that $p_X(p_Y^{-1}(f(\alpha)))$ contains both $\alpha$ and $g \circ f(\alpha)$. Hence
\begin{align*}
\size_{X^*}(g \circ f (\alpha) \cup \alpha) &\leq \size_{X^*}(p_X(p_Y^{-1}(f(\alpha))))\\
&\leq \size_{Y^*}(p_Y(p_Y^{-1}(f(\alpha))))+\epsilon \\
&= \size_{Y^*}(f(\alpha))+\epsilon\\
&\leq \size_{X^*}(\alpha)+2\epsilon.
\end{align*}
This shows that $$\codegree^\infty(g \circ f, \iden_{X^*})\leq \codegree(g \circ f, \iden_{X^*}) \leq 2\epsilon. $$. Similarly, $$\codegree^\infty(f \circ g,\iden_{Y^*}) \leq 2\epsilon.$$ This completes the proof.
\end{proof}

\begin{proof}[Proof of Theorem \ref{stb}]
By the definition of interleavings for filtered simplicial complexes and Remark \ref{codegreeinfty}, an $\epsilon$-interleaving between filtered simplicial complexes induces an $\epsilon$-interleaving between their persistence modules. Hence $$\intp(\PH_k(X^*),\PH_k(Y^*)) \leq \intf(X^*,Y^*).  $$
Now let $R$ be a correspondence between the vertex sets of $X^*,Y^*$. Then there are morphism $f: X^*\to Y^*,g: Y^* \to X^*$ such that $R$ contains graphs of $f$ and $g$. By Lemma \ref{dis-int} $X^*,Y^*$ are $\dis(R)$-interleaved. Since $R$ was an arbitrary correspondence, by Remark \ref{correspondence} $$\intf(X^*,Y^*) \leq 2\dgh(X^*,Y^*) .$$
\end{proof}

\section{The vertex quasi-distance and simplification }\label{simplification}
We start by giving a definition.
\begin{definition}[The vertex quasi-distance]\label{def:vqd}
Let $X^*$ be a filtered simplicial complex. Given vertices $v,w$ of $X^*$, define the \textit{\deltainv} $\delta_X(v,w)$ to be the minimal $\delta \geq 0$ such that $$\size_X(\alpha \cup \{v\}) + \delta \geq \size_X(\alpha \cup \{w\}), $$ for each non-empty set of vertices $\alpha$.
\end{definition}
Note that taking $\alpha$ as the full vertex set already requires $\delta \geq 0$, hence we can equivalently define $\delta_X(v,w)$ by $\delta_X(v,w):=\max_\alpha \big(\size_X(\alpha \cup \{w\}) - \size_X(\alpha \cup \{v\})\big).$

Although the \deltainv  is not necessarily symmetric (i.e. $\delta_X(v,w)$ may be different from $\delta_X(w,v)$), the following remark shows that it satisfies other properties of a metric. Such structures are called \textit{quasimetric spaces}.

\begin{Remark}[Quasimetric]\label{deltapair} For all vertices $v,v',v''$, 
\begin{enumerate}
\item $\delta_X(v,v)=0.$
\item $\delta_X(v,v')+\delta_X(v',v'') \geq \delta_X(v,v'').$
\end{enumerate}
\end{Remark}

\begin{example}[The case of Vietoris-Rips complexes]\label{ex:rips-ms}
Let $X^*$ be the Vietoris-Rips complex of a finite metric space $(M,d_M)$. Let us show that $\delta_X(x,y)=d_M(x,y).$ Recall that in this case the size function is the diameter. Note that for $\alpha=\{x\}$, $\diam_M(\alpha \cup \{y\}) - \diam_M(\alpha \cup \{x\})=d_M(x,y)-0=d_M(x,y),$
hence  $\delta_X(x,y) \geq d_M(x,y).$ Note that, by triangle inequality for any $z$ we have $d_M(y,z) \leq d_M(x,z) + d_M(x,y)$ and this implies that for any subset $\alpha$ we have $\diam_M(\alpha \cup \{y\}) \leq \diam_M(\alpha \cup \{ x\}) + d_M(x,y)$ and this implies that $\delta_X(x,y) \leq d_M(x,y).$ Hence $\delta_X(x,y)=d_M(x,y).$
\end{example}

\begin{definition}[Codensity function] For each vertex $v$ let
\begin{itemize}
\item $\delta_X(v):=\min_{w \neq v}\delta_X(v,w).$ This is called the codensity of vertex $v$.
\item $\delta(X^*):=\min_v \delta_X(v)$ (minimal codensity of $X^\ast$). 
\end{itemize}
\end{definition}

We introduce this invariant to control the effect of removing a vertex from a filtered simplicial complex on its persistent homology. Before proving Proposition \ref{removal} let us  precisely define what we mean by removing a vertex. 

\begin{definition}[Filtered subcomplex]
A \textit{filtered subcomplex} of a filtered simplicial complex $X^*$ is a filtered simplicial complex $Y^*$ such that for each $t$, $Y^t$ is a subcomplex of $X^t$. We call $Y^*$ a \textit{full} filtered subcomplex if each $Y^i$ is a full subcomplex of $X^t$, precisely a simplex of $X^t$ is a simplex of $Y^t$ if and only if its vertices are in $Y^t$. Note that a full subcomplex is determined by its vertex set. Therefore, if we take a vertex $v$ from a filtered simplicial complex $X^*$ with vertex set $V$, there exists a unique full filtered subcomplex of it such that the vertex set at index $t$ is the vertex set of $X^t$ minus $v$. We denote this subcomplex by $(X-\{v\})^*$.
\end{definition}

\begin{Remark}[Restriction]
The size function of $(X-\{v\})^*$ is the restriction of the size function of $X^*$.
\end{Remark}

\begin{proof}[Proof of Proposition \ref{removal}]
Let $w \neq v$ be a vertex such that $\delta_X(v,w)=\delta_X(v)$. Let $f:X^* \to (X-\{v\})^*$ be the map which is identity on all vertices except $v$ and maps $v$ to $w$. Let $\iota:(X-\{v\})^* \to X^*$ be the inclusion map. Let us show that $(f,\iota)$ is a $\delta_X(v)$-interleaving. We have $\degree(\iota)=0$. Let $\alpha$ be a non-empty subset of the vertex set of $X^*$. If $v \notin \alpha$ then $f(\alpha)=\alpha$. If $v \in \alpha$, then $f(\alpha) \subseteq \alpha \cup \{w\}$, hence $$\size_X(f(\alpha))\leq \size_X(\alpha \cup \{w\}) \leq \size_X(\alpha \cup \{v\})+\delta_X(v) =\size_X(\alpha)+\delta_X(v). $$ Hence $\degree(f)\leq \delta_X(v)$. Since $f \circ \iota = \iden_{(X-\{v\})^*}$, $\codegree^\infty(f \circ \iota, \iden_{(X-v)^*})=0$. If $v \notin \alpha$, then $\alpha \cup \iota \circ f (\alpha)=\alpha$. If $v \in \alpha$ then $\alpha \cup \iota \circ f (\alpha)=\alpha \cup \{w\}$, hence $$\size_X(\alpha \cup \iota \circ f (\alpha))=\size_X(\alpha \cup \{w\}) \leq \size_X(\alpha \cup \{v\})+\delta_X(v)=\size_X(\alpha)+\delta_X(v). $$ Hence $\codegree^\infty(\iota \circ f,\iden_{X^*})\leq \delta_X(v)$. Therefore $X^*,(X-\{v\})^*$ are $\delta_X(v)$-interleaved.
\end{proof}

 \paragraph{Computational consequences.}
Note that $\delta_X(v,w)$ can only become smaller after we remove a vertex since the maximum in the definition of \deltainv is now taken on a smaller set. However, this does not imply that $\delta_X(v)$ also become smaller after removing a vertex, since it is possible that the removed vertex $w$ is the vertex realizing $\delta_X(v)=\delta_X(v,w)$. Still, the observation of the monotonicity of $\delta_X(v,w)$ gives us a method to simplify a filtered simplicial complex while bounding the approximation error in the persistent homology. Let us enumerate the vertex set of $X^*$ as $(v_1,\dots,v_n)$ and let $Q(X^*)$ be the $n \times n$ matrix given by $[\delta_X(v_i,v_j)]_{i,j}$. This method is streamlined in Listing \ref{list:simp}. Note that when the procedure terminates, the interleaving distance $\intf(X^*,Y^*) \leq \mathrm{errorBound}.$ In the following subsections we discuss how to decrease the time complexity and/or error bound obtained from this method, if we are only interested in certain degrees of homology.

\begin{lstlisting}[caption={Simplification Method. Note: at the end of the execution \texttt{Y*} is full subcomplex of \texttt{X*} with vertex set \texttt{W}. Note: the procedure \texttt{ComputeCodensityMatrix()} is discussed in the next section.},label=list:simp,captionpos=t,float,abovecaptionskip=-\medskipamount]
INPUT: X*, N: number of vertices to be removed
OUTPUT: Y*, a full subcomplex and errorBound
SET Q=ComputeCondensityMatrix(X*), errorBound=0, W=Vertex set of X*.
for k from 1 to N
   (i,j)=index of the minimal nondiagonal element of Q
   errorBound=errorBound + Q(i,j)
   remove W(i) from W
   remove i-th row and column from Q
endfor
\end{lstlisting}

\subsection{Computating $\delta_X(v,w)$: The procedure \texttt{ComputeCodensityMatrix()}} The simplification method given in Listing \ref{list:simp} calls the procedure \texttt{ComputeCodensityMatrix()} which calculates the matrix $[\delta_X(v_i,v_j)]_{i,j}$. In this section we explain the mathematical ideas behind it. We will not provide pseudo-code as the procedure will be made evident.

Normally, the definition of $\delta_X(v,w)$ (Definition \ref{def:vqd}) requires  checking all non-empty subsets of the vertex set $V$ of $X^*$, which in total gives us a complexity of $O(2^n)$. However, we can achieve a better complexity  if our filtered simplicial complex has some structure. 

Proposition \ref{clique-delta} below shows that that if the filtered simplicial complex is clique, then we only need to check  singletons in order to calculate $\delta_X(v,w)$. Recall that a simplicial complex is called clique if a simplex is included in it whenever its 1-skeleton is included. A filtered simplicial complex $X^*$ is called clique if each $X^t$ is clique.

\begin{proposition}[The case of clique filtered simplicial complexes]\label{clique-delta}
Let $X^*$ be a clique filtered simplicial complex with  vertex set $V$ and  size function $\size_X$. Then $$\delta_X(v,v')= \max_{w \in V} \big(\size_X(\{v',w \}) - \size_X(\{v,w \})\big).$$
\end{proposition}

We have the following lemma whose proof we omit:
\begin{lemma}\label{clique-size}
Let $X^*$ be a clique filtered simplicial complex with the size function $\size_X$. Then, 
$\size_X(\alpha)=\max_{v,w \in \alpha} \size_X(\{v,w\}).$
\end{lemma}

\begin{proof}[Proof of Proposition \ref{clique-delta}]
Let $\epsilon:=\max_{w \in V} \big(\size_X(\{v',w \}) - \size_X(\{v,w \})\big)$. Recall that 
\[\delta_X(v,v')=\max_{\alpha \subseteq V, \alpha \neq \emptyset} \size_X(\alpha \cup \{v'\})-\size_X(\alpha \cup \{v\}), \]
hence $\delta_X(v,v') \geq \epsilon$. Let us show that $\delta_X(v,v') \leq \epsilon.$ Let $\alpha$ be a non-empty subset of $V$. Then by Lemma \ref{clique-size} we have
\begin{align*}
\size_X(\alpha \cup \{v'\}) &=\max( \max_{w,w' \in \alpha} \size_X(\{w,w'\}), \max_{w \in \alpha} \size_X(\{v',w\}) )\\
												&\leq \max( \max_{w,w' \in \alpha} \size_X(\{w,w'\}), \max_{w \in \alpha} \size_X(\{v,w\}) + \epsilon )\\
                                                &\leq \max( \max_{w,w' \in \alpha} \size_X(\{w,w'\}), \max_{w \in \alpha} \size_X(\{v,w\}) )+ \epsilon \\
                                                &=\size_X(\alpha \cup \{v\})+\epsilon.
\end{align*}
Since $\alpha$ was arbitrary, $\delta_X(v,v') \leq \epsilon.$
\end{proof}

\begin{definition}As a generalization of the concept of a clique complex, let us call a simplicial complex $k$-clique if a simplex is contained in it if and only if its $k$-skeleton is contained in it. A clique complex is $1$-clique with respect to this definition.
\end{definition}
By a proof similar to that of Proposition \ref{clique-delta}, we can obtain the following generalization:

\begin{proposition}The case of $k$-clique filtered simplicial complexes\label{k-clique-delta}
Let $k$ be a positive integer and let $X^*$ be a filtered simplicial complex such that for each $t$ $X^t$ is $k$-clique. Then for all vertices $v$ and $v'$,
$$\delta_X(v,v')=\max_{\alpha, 0 < |\alpha| \leq k} \big(\size_X(\alpha \cup \{v'\})-\size_X(\alpha \cup \{v\})\big).$$
\end{proposition}

Now we use this to show we can turn a given filtered simplicial complex into one satisfying the assumptions in Proposition \ref{k-clique-delta} without losing persistent homology information in degrees less than $k$.

\begin{proposition}\label{homology-less}
Let $X^*$ be a filtered simplicial complex. Let $Y^*$ be the filtered simplicial complex with the same vertex set as $X^*$ such that for each $t$, a simplex is in $Y^t$ if and only if its $k$-skeleton is in $X^t$. Note that $Y^*$ is well defined and $X^t \subseteq Y^t$ for each $t$. We have:
\begin{enumerate}
\item[(1)]  $Y^t$ is $k$-clique for all $t$.
\item[(2)] $\PH_{<k}(X^*) \cong \PH_{<k}(Y^*).$
\item[(3)] $\size_Y(\alpha)=\max_{\beta \subseteq \alpha, 0<|\beta|\leq k+1} \size_X(\beta)$.
\end{enumerate}
\end{proposition}
\begin{proof}
(1) Assume the $k$-skeleton $\alpha^k$ of $\alpha$ is contained in $Y^i$. Then the $k$-skeleton of $\alpha^k$, which is $\alpha^k$ itself is contained in $X^i$. Therefore $\alpha$ is contained in $Y^i$.

\noindent (2) Note that if $\alpha$ is a simplex of dimension less than or equal to $k$, then its $k$-skeleton is itself, hence it is contained in $X^i$ if and only if it is contained in $Y^i$. Therefore the inclusion $X^* \to Y^*$ is identity in the level of $k$-skeleton. Therefore, it induces an isomorphism between homology groups of degree less than $k$.

\noindent (3) Let $r:=\max_{\beta \subseteq \alpha, 0<|\beta|\leq k+1} \size_X(\beta).$ Since the $k$-skeletons of $X^*,Y^*$ are the same, for $|\beta| \leq k+1$ we have $\size_X(\beta)=\size_Y(\beta)$. Therefore
\[\size_Y(\alpha) \geq \max_{\beta \subseteq \alpha, 0<|\beta|\leq k+1} \size_Y(\beta) = \max_{\beta \subseteq \alpha, 0<|\beta|\leq k+1} \size_X(\beta) = r. \]
Now let us show that $\size_Y(\alpha) \leq r.$ We need to show that $\alpha \in Y^r$. By the definition of $r$, the $k$-{th} skeleton of $\alpha$ is contained in $X^r$. This implies that $\alpha$ is in $Y^r$.
\end{proof}

\subsection{Specializing  $\delta_X(v,w)$ according to homology degree}
In this section we refine our ideas so that given a filtered simplicial complex $X^*$ and $k\in\mathbb{N}$, the bound given in Proposition \ref{removal} is better adapted to scenarios when one only wishes to compute persistent homologies $\PH_j(X^*)$ for $j\geq k$. 

Given a filtered simplicial complex and $k\in\mathbb{N}$, let $Y^*=\mathrm{T_k(X^*)}$ be the filtered simplicial complex with the same vertex such that a simplex is in $Y^t$ if it is in $X^t$ and each simplex in its $k$-skeleton is contained in a $k$-simplex of $X^t$. In other words, we remove simplices from $X^t$ which have dimension less than $k$ and are not contained in any $k$-simplex of $X^t$. Note that $Y^*$ is well defined and $Y^t \subseteq X^t$ for each $t$.

\begin{proposition}\label{lower}
 Denote the vertex quasi-distance for $X^*$  by $\delta_X$, and by $\delta_Y$ denote the vertex quasidistance of $Y^*=\mathrm{T}_k(X^*)$. Then, we have:
\begin{itemize}
\item[(1)] $\PH_{\geq k}(Y^*)=\PH_{\geq k}(X^*)$.
\item[(2)] $\size_Y(\alpha)=\min_{\beta \supseteq \alpha,|\beta|\geq k+1} \size_X(\beta).$
\item[(3)] $\delta_Y(v,w) \leq \delta_X(v,w).$
\item[(4)] Let $m\geq k$ be a non-negative integer. If $X^*$ satisfies the property that for each $t$, $\alpha\in X^t$ if and only if the $m$-skeleton of $\alpha$ is in $X^t$, then $Y^*$ satisfies this property too.
\end{itemize}
\end{proposition}

\begin{proof}

\noindent (1) Note that for $k' \geq k$, the $k'$-simplices of $X^i$ is same with the $k'$-simplices of $Y^i$. Since the homology of degree $k'$ is determined by such cells, the inclusion $Y^* \subseteq X^*$ induces the isomorphism $\PH_{\geq k}(Y^*) \to \PH_{\geq k}(X^*).$ 

\noindent (2) By the identity mentioned above, for $\beta$ with $|\beta|\geq k+1$ we have $\size_Y(\beta)=\size_X(\beta)$. Let us denote $r:=\min_{\beta \supseteq \alpha,|\beta|\geq k+1} \size_X(\beta).$ Then we have
\[\size_Y(\alpha) \leq \min_{\beta \supseteq \alpha,|\beta|\geq k+1} \size_Y(\beta)=\min_{\beta \supseteq \alpha,|\beta|\geq k+1} \size_X(\beta)=r. \]Let us show that $\size_Y(\alpha) \geq r$. Let $s<r$. Let us show that $\size_Y(\alpha)>s$. If the dimension of $\alpha$ is greater than or equal to $k$, then $\size_Y(\alpha)=\size_X(\alpha)=r>s$. Now assume that $\alpha$ has dimension less than $k$. By definition of $r$, every $k$-cell containing $\alpha$ has size strictly greater than $s$, therefore $\alpha$ is not contained in $X^s$. Hence, we have $\size_Y(\alpha) > s$. Since $s<r$ was arbitrary $\size_Y(\alpha) \geq r$.

\noindent (3) Let $\alpha$ be a simplex and $\beta$ be the simplex with dimension greater than or equal to $k$ containing $\alpha \cup \{v\}$  such that $\size_Y(\alpha \cup \{v\})=\size_X(\beta)$.  Then we have
\begin{align*}
\size_Y(\alpha \cup \{w\}) - \size_Y(\alpha \cup \{v\}) &\leq \size_Y(\beta \cup \{w\}) - \size_Y(\alpha \cup \{v\}) \\
																											&=\size_X(\beta \cup \{w\}) - \size_X(\beta) \\
                                                                                                            &=\size_X(\beta \cup \{w\}) - \size_X(\beta \cup \{v\})\\
                                                                                                            &\leq \delta_X(v,w).
\end{align*}
Since $\alpha$ was arbitrary, $\delta_Y(v,w) \leq \delta_X(v,w).$

\noindent (4) Let $\alpha$ be a simplex whose $m$-skeleton $\alpha^m$ is in $Y^t$. Let us show that $\alpha \in Y^t$. Note that $\alpha$ is in $X^t$. If the dimension of $\alpha$ is greater or equal than $k$, then $\alpha$ is in $Y^t$. Now assume that the dimension of $\alpha$ is less than $k$. Then we have that $\alpha^m=\alpha$ is in $Y^t$. 
\end{proof}

\noindent
\textbf{Summary.} If we are only interested in degree $k$ persistent homology we can first apply the \textit{clique-fication} process described in Proposition \ref{homology-less} for $k+1$ so that the calculation of each entry $\delta_X(v,w)$ of the matrix $Q(X^*)$ becomes a $O(n^k)$ task instead of $O(2^n)$, where $n$ is the number of vertices. Then we can apply the process described in Proposition \ref{lower} so that we lower the values of $\delta_X(v,w)$ and get a better error bound for the simplification process. Then we can start our simplification process.  After removing a vertex we have two options, we can either keep working with the original codensity matrix to get the upper bound on the change in persistent homology, or we may want to compute the codensity matrix again, since its elements may decrease after removal. Note that if we remove a vertex from a $k$-clique filtered simplicial complex, it will still be $k$-clique. Hence calculating the codensity matrix does not become more costly after removing a vertex.

\subsection{An application to the Vietoris-Rips filtration of finite metric spaces}\label{sec:vr} 

Let $X^*$ be the Vietoris-Rips complex of a finite metric space $(M,d_M)$. Removing a vertex in this case means passing to the sub-metric space $M-\{x\}$ for some $x$ in $M$. The Gromov-Hausdorff cost of this removal is at least half of the minimal positive distance in $M$. Furthermore, any correspondence between $M$ with $M-\{x\}$  has distortion $\min_{y \in M-\{x\}} d_M(x,y)$. By Theorem \ref{stb}, this bound gives an upper bound for the interleaving distance between the corresponding persistence modules. Let us see how we can improve this bound by applying methods mentioned in this section if we are only interested in $\PH_{k \geq 1}(\mathrm{VR}^*(X))$.

In Example \ref{ex:rips-ms}, we have seen that for $x,y \in M$, $\delta_X(x,y)=d_M(x,y)$, hence $\delta_X(x)$ coincides with the minimal distance to $x$, which is not better than the Gromov-Hausdorff cost. However, we know that we can possibly decrease the value of $\delta_X(x,y)$ by methods described in Proposition \ref{lower}. Let us denote the modified size function described in Proposition \ref{lower} by $\diam_{M,k}$. More precisely, for $\alpha \subseteq M$, we have $$\diam_{M,k}(\alpha):=\min_{\beta \supseteq \alpha,|\beta|\geq k+1} \diam_M(\beta).$$ Note that $\diam_{M,0}=\diam_M$. Let us denote the corresponding filtered simplicial complex by $X^*_k$.  By Proposition \ref{lower}, the persistent homology of  $X^*=\mathrm{VR}(M)$ is the same as that of $X_k^*$. Therefore, if we are interested in persistence homology of degree at least $1$, then instead of working with the Vietoris-Rips complex, we can work with $X_1^*$ which has the advantage of having a smaller codensity function. 

Since the Vietoris-Rips is clique, by Proposition \ref{lower} $X_1^*$ is also clique. Let $\delta_1$ denote the codensity function of $X^*_1$. By Proposition \ref{clique-delta}, we have
$\delta_{1}(x,y)=\max_{p \in M}\big( \diam_1(\{y,p\})-\diam_1(\{x,p\}) \big).$ Let us give a proposition which we use in the example following it to show that $\delta_1(x)=\min_{y\neq x}\delta_1(x,y)$ can be much smaller than $\delta_X(x)$.

\begin{proposition}\label{metric}
Let $M$  be a metric space and $\delta_1$ be the codensity function of $X_1^*$, described as above. Let $x$ be a point in $X$ and $y$ be the closest point to $x$. Then $$ \delta_1(x,y)=\max\Big(0, \max_{p \neq x,y}\big(d_M(y,p)-d_M(x,p)\big)\Big).$$
\end{proposition}
\begin{proof}
\begin{align*}
\delta_1(x,y)&=\max_{p \in M} \big(\diam_{M,1}(\{p,y\})-\diam_{M,1}(\{p,x\})\big)\\
						 &= \max\Big(d_M(x,y)-\diam_{M,1}(\{x\}), \diam_{M,1}(\{ y\}) - d_M(x,y), \max_{p \neq x,y}d_M(p,y)-d_M(x,y)\Big)\\
                         &=\max\Big(0,\max_{p\neq x,y} \big(d_M(y,p)-d_M(x,p)\big)\Big).
\end{align*}
\end{proof}

\begin{example}[Circle with flares]\label{ex:flares}
Let $M$ be a finite finite metric space described as follows: It is a finite set of points selected from a circle and some flares attached to it, see Figure \ref{fig:flares}. Let us show that for an endpoint $x$ of a flare in $M$, $\delta_1(x)=0$. Note that this implies that our method (see Listing \ref{list:simp}) will inductively remove all points in flares without any cost on $\PH_{\geq 1}(\VR^*(M))$ until only the points on the circle are left. Note that this is significantly less than both the Gromov-Hausdorff distance between the original space $M$ and the final space $M'$, and the sum of Gromov-Hausdorff costs of succesively removing single points.

Let $y$ be the closest to point to $x$ in $M$. Since $x$ is a endpoint in a flare, for each $p \neq x$ we have $d_M(x,p)=d_M(x,y)+d_M(y,p)$, in particular $d_M(y,p) \leq d_M(x,p)$. Therefore, by Proposition \ref{metric} we have 
$\delta_1(x) \leq \delta_1(x,y) = 0. $
\end{example}

\section{Classification of filtered simplicial complexes via $d_\mathrm{I}^\mathrm{F}$}\label{classification}
In this section we prove Theorem \ref{cls}.
\begin{definition}[Simple filtered simplicial complex]
A filtered simplicial complex $X^*$ is called \textit{simple} if $\delta(X^*)>0$.
\end{definition}

\begin{lemma}[Non-identity morphisms]\label{delta-identity}
Every non-identity morphism $f:X^* \to X^*$ has $\codegree^\infty(f,\iden_{X^*}) \geq \delta(X^*)$. 
\end{lemma}
\begin{proof}
Let $\iden_{X^*}=f_0,f_1,\dots,f_n=f$ be a family of morphisms realizing $\delta:=\codegree^\infty(f,\iden_{X^*})$. Without loss of generality, we can assume that $f_1$ is non-identity. Note that 
$\codegree(f_1,\iden_{X^*}) \leq \delta$. Let $v$ be a vertex such that $w:=f_1(v)$ different from $v$. Now, we have
\begin{align*}
\size_X(\alpha \cup \{w\})&\leq \size_X\big((\alpha \cup \{v\}) \cup (f_1(\alpha)\cup \{w\})\big)\\
						&\leq \size_X(\alpha \cup \{v\})+\delta.
\end{align*}
Since $\alpha$ was arbitrary, $$\codegree^\infty(f,\iden_{X^*})=\delta \geq \delta_X(v,w) \geq \delta(X^*). $$
\end{proof}

\begin{proposition}\label{gh-intf}
Let $X^*,Y^*$ be simple filtered simplicial complexes such that for some $r\geq 0$, $\min\big(\delta(X^*),\delta(Y^*)\big) > r$. If $\intf(X^*,Y^*) \leq r/2$, then $2\,\dgh(X^*,Y^*) = \intf(X^*,Y^*).$ Furthermore, in this case there exists an invertible morphism $f: X^* \to Y^*$ with inverse $g:Y^* \to X^*$ such that the value above is equal to $\max(\degree(f),\degree(g)).$
\end{proposition}
\begin{proof}
By Theorem \ref{stb}, we already know that $2\dgh(X^*,Y^*) \geq \intf(X^*,Y^*)$. Let us show that $2\dgh(X^*,Y^*) \leq \intf(X^*,Y^*)$.

Let $f: X^* \to Y^*$, $g: Y^* \to X^*$ be morphisms realizing the interleaving distance $\epsilon:=\intf(X^*,Y^*)$. Then, $$\codegree^\infty(g \circ f,\iden_{X^*}) \leq 2\epsilon \leq r < \delta(X^*), $$
hence by Lemma \ref{delta-identity} $g \circ f = \iden_{X^*}$. Similarly $f \circ g=\iden_{Y^*}$. Note that this implies that $\epsilon=\max(\degree(f),\degree(g))$. If we define $R$ as the graph of $f$, then $R$ is a correspondence between the vertex sets of $X^*,Y^*$. It is enough to show that $\dis(R) \leq \epsilon$.

Let $\beta$ be a non-empty subset of $R$. Let us denote the projection maps from $R$ to the vertex sets of $X^*,Y^*$ by $p_X,p_Y$ respectively. Let $\alpha:=p_X(\beta)$. Since $R$ is the graph of $f$, $p_Y(\beta)=f(\alpha)$. Now we have,
\begin{align*}
\size_{Y^*}(p_Y(\beta))-\size_{X^*}(p_X(\beta))&=\size_{Y^*}(f(\alpha)) - \size_{X^*}(\alpha)\\
&\leq \degree(f) \leq \epsilon,
\end{align*}
and
\begin{align*}
\size_{X^*}(p_X(\beta))-\size_{Y^*}(p_Y(\beta))&=\size_{X^*}(\alpha) - \size_{Y^*}(f(\alpha))\\
&= \size_{X^*}(g(f(\alpha))) - \size_{Y^*}(f(\alpha)) \\
&\leq \degree(g) \leq \epsilon,
\end{align*}
hence
\[|\size_{X^*}(p_X(\beta))-\size_{Y^*}(p_Y(\beta))| \leq \epsilon. \]
Since $\beta$ was arbitrary, we have
\[ 2\dgh(X^*,Y^*) \leq \dis(R) \leq \epsilon = \intf(X^*,Y^*). \]
\end{proof}

\begin{proof}[Proof of Theorem \ref{cls}]

\noindent\textbf{Existence:} By Proposition \ref{removal}, by removing $v$ such that $\delta_X(v)=0$ one by one, we get a simple filtered simplicial complex $C^*$ such that $X^*$ is equivalent to $C^*$, i.e. $\intf(X^*,C^*)=0$. Note that for a filtered simplicial complex $P^*$ with a single vertex, $\delta_X(P^*)=\infty$ hence it is simple. Since $C^*$ is obtained from $X^*$ by removing vertices, it is a full subcomplex.

\noindent
\textbf{Uniqueness:} Assume $C^*,T^*$ are simple filtered simplicial complexes equivalent to $X^*$. Then, by the triangle inequality for $\intf$ they are equivalent to each other. Hence by Proposition \ref{gh-intf}, taking $r=0$, we see that $C^*,T^*$ are isomorphic, since the map $f$ becomes a size preserving bijection as both $f$ and its inverse has degree $0$.
\end{proof}

\begin{Remark}[Cores and isomorphism]
As it is explained in the proof above, we obtain the core of $X^*$ by removing vertices $v$ with $\delta_X(v)=0$ one by one. Since the core is determined up to isomorphism, the order in which we remove the points does not matter, in any case we remove the same number of points and although we may reach different subcomplexes, they will be necessarily isomorphic.
\end{Remark}

Theorem \ref{cls} implies the following. Let $\mathcal{C}(X)=\{C^*|\,\intf(X^*,C^*)=0\},$ that is, $\mathcal{C}(X^*)$ contains all the filtered simplicial complexes equivalent to $X^*$, and in particular, it contains all those with the same persistent homology as $X^*$. Let $m(X^*)$ be the \emph{minimal possible cardinality} over all vertex sets of elements in $\mathcal{C}(X^*).$  

\begin{corollary}[The core is minimal]\label{coro:min-core}
The vertex set of the core $\skeleton{X^*}$ has minimal cardinality $m(X^*)$.
\end{corollary}
\begin{proof}
Let $C^*\in\mathcal{C}(X^*)$ be such that  its vertex sets has minimal cardinality $m(X^*)$. It follows that $C^*$ is simple for otherwise, according to Proposition \ref{removal}, we would be able to reduce its size. Then, by the triangle inequality for $\intf$  (Proposition \ref{interleavingdistance}) and Theorem \ref{cls}, the distance between $\skeleton{X^*}$ and $C^*$ is also zero. But since both $C^*$ and $\skeleton{X^*}$ are simple, Theorem \ref{cls} implies that they have to be isomorphic. In particular, their vertex sets ought to have the same cardinality.
\end{proof}

\section{An example where $\intf \ll \dgh$}\label{example}

Let $X^*$ be a filtered simplicial complex with the size function $\size_X$. Given a vertex $w$ and a real number $r \geq 0$, we define a \textit{single vertex extension} $X_{w,r}^*$ as follows. The underlying vertex set is the vertex set of $X^*$ plus a new vertex $v_0$. We define a size function $\tilde{\size}$ on this new vertex set as follows. We set $\tilde{\size}(\{v_0\}):=\size_X(\{w\})+r$ and for a nonempty subset $\alpha$ of the vertex set of $X^*$, we set
\[\tilde{\size}(\alpha):=\size_X(\alpha), \quad \tilde{\size}(\alpha \cup \{v_0\}):=\size_X(\alpha \cup \{w\})+r. \]
Let us show that $\tilde{\size}$ is monotonic with respect to inclusion. Let $\alpha$ be non-empty subset of the vertex set of $X_{r,w}^*$. If $v_0 \in \alpha$, then $\tilde{\size}(\alpha \cup \{ v_0\})=\tilde{\size}(\alpha)$. If $v_0 \notin \alpha$, then $$\tilde{\size}(\alpha \cup \{v_0 \})=\size_X(\alpha \cup \{w\}) + r \geq \size_X(\alpha) = \tilde{\size}(\alpha).$$
Hence, in any case $\tilde{\size}(\alpha \cup \{v_0 \}) \geq \tilde{\size}(\alpha)$. Now let $\alpha \subseteq \beta$. If $v_0 \notin \beta$, then $$\tilde{\size}(\beta)=\size_X(\beta)\geq \size_X(\alpha) = \tilde{\size}(\alpha).$$ If $v_0 \in \beta$, then $$\tilde{\size}(\beta) = \size_X(\beta \cup \{w\} - \{ v_0\} )+r \geq \size_X(\alpha \cup \{w\} - \{v_0\}) + r = \tilde{\size}(\alpha \cup \{v_0\}) \geq \tilde{\size}(\alpha).$$
Hence $\tilde{\size}$ is a size funtion and $X_{w,r}^*$ is a filtered simplicial complex. Note that $X^*$ is a full subcomplex of $X_{w,r}^*$ obtained by removing the vertex $v_0$. Let us show that $\delta_{X_{w,r}}(v_0)=0.$ For any non-empty subset $\alpha$ of the vertex set of $X_{w,r}^*$, we have
\[\tilde{\size}(\alpha \cup \{w\}) \leq \tilde{\size}(\alpha \cup \{v_0,w\})=\size_X(\alpha-\{v_0\} \cup \{w\})+r=\tilde{\size}(\alpha \cup \{v_0\})\]
Hence $\delta_{X_{w,r}}(v_0,w)=0$, so $\delta_{w,r}(v_0)=0$. By Proposition \ref{removal}, $\intf(X_{w,r}^*,X^*)=0$.

Now let us show that $\dgh(X_{w,r}^*,X^*) \geq r/2$. Let $(Z,p,\tilde{p})$ be any tripod between the vertex sets $V,\tilde{V}$ of $X^*,X_{w,r}^*$. Then
\[\dis(Z) \geq \tilde{\size}(\tilde{p}(Z))-\size_X(p(Z))=\tilde{\size}(\tilde{V})-\size_X(V)=r. \]
Since $Z$ was arbitrary, $\dgh(X_{w,r}^*,X^*)\geq r/2$. Therefore, if we take $r \gg 0$, then
\[ \intf(X_{w,r}^*,X^*)=0 \ll r/2 \leq \dgh(X_{w,r}^*,X^*). \]

\begin{Remark}\label{simplex-generalization}
Recall $\Delta_n^*$ from Example \ref{simplex-star}, with the vertex set $\{0,\dots,n\}$ and the size function given by maximum. Note that $\Delta_{n+1}^*=(\Delta_n^*)_{w=n,r=1}$. This also shows that $\intf(\Delta_m^*,\Delta_n^*)=0$.
\end{Remark}

\newcommand{\etalchar}[1]{$^{#1}$}

\end{document}